\documentclass[12pt,english]{amsart}
\usepackage[T1]{fontenc}
\usepackage[utf8]{inputenc}
\usepackage{lmodern}
\usepackage[a4paper]{geometry}

\usepackage{enumerate}
\usepackage{amsmath}
\usepackage{amssymb}
\usepackage{amsthm}
\usepackage{dsfont}
\usepackage{graphicx}
\usepackage{xcolor}

\usepackage{babel}

\newtheorem{theorem}{Theorem}[section]

\newtheorem*{theorem*}{Theorem}

\newtheorem*{notation}{Notation}
\newtheorem{proposition}[theorem]{Proposition}
\newtheorem{lemma}[theorem]{Lemma}
\newtheorem{corollary}[theorem]{Corollary}
\newtheorem{Definition}[theorem]{Definition}
\newenvironment{definition}{\begin{Definition}\rm}{\end{Definition}}
\newtheorem{Remark}[theorem]{Remark}
\newenvironment{remark}{\begin{Remark}\rm}{\end{Remark}}
\newtheorem{Example}[theorem]{Example}
\newenvironment{example}{\begin{Example}\rm}{\end{Example}}

\theoremstyle{definition}
\newtheorem{question}{Question}

\newcommand{\C}{\mathbb{C}}

\newcommand{\N}{\mathbb{N}}
\newcommand{\R}{\mathbb{R}}

\renewcommand{\AA}{\mathcal{A}}
\newcommand{\KK}{\mathcal{K}}
\renewcommand{\SS}{\mathcal{S}}

\renewcommand{\tilde}{\widetilde}

\newcommand{\indi}[1]{\mathds{1}_{#1}}

\begin{document}

	\title[$\mathcal{U}$-Frequent hypercyclicity notions and related weighted densities]{$\mathcal{U}$-Frequent hypercyclicity notions and related weighted densities}
	
	\author{R. Ernst, C. Esser, Q. Menet}

	\address{Romuald Ernst, LMPA, Centre Universitaire de la Mi-Voix, Maison de la Recherche Blaise-Pascal, 50 rue Ferdinand Buisson, BP 699, 62228 Calais, France}
	\email{ernst.r@math.cnrs.fr}

	\address{Céline Esser, Université de Liège - Institut de Mathématique, Allée de la Découverte 12, 4000 Liège, Belgique}
	\email{celine.esser@uliege.be}

	\address{Quentin Menet, Laboratoire de Mathématiques de Lens, Université d'Artois, Rue Jean Souvraz SP 18, 62307 Lens, France}
	\email{quentin.menet@univ-artois.fr}

\thanks{The first and the third author were supported by the grant ANR-17-CE40-0021 of the French National Research Agency ANR (project Front) and by Programme PEPS JC 2018 INSMI}

\begin{abstract}
We study dynamical notions lying between $\mathcal{U}$-frequent
hypercyclicity and reiterative hypercyclicity by investigating
weighted upper densities between the unweighted upper density and the
upper Banach density. While chaos implies reiterative hypercyclicity,
we show that chaos does not imply $\mathcal{U}$-frequent
hypercyclicity with respect to any weighted upper density. Moreover, we
show that if $T$ is $\mathcal{U}$-frequently hypercyclic
(resp. reiteratively hypercyclic) then the n-fold product of $T$ is
still $\mathcal{U}$-frequently hypercyclic (resp. reiteratively
hypercyclic) and that this implication is also satisfied for each of
the considered $\mathcal{U}$-frequent hypercyclicity notions.

\end{abstract}

\maketitle

Let $X$ be a separable infinite-dimensional Banach space and $T$ be a linear continuous operator on $X$. The operator $T$ is said to be hypercyclic if there exists a vector $x\in X$ having a dense orbit under the action of $T$ i.e. $\text{Orb}(x,T):=\{T^nx:n\in\N\}$ is dense in $X$. Such a vector $x$ is said to be hypercyclic for $T$ and we denote by $HC(T)$ the set of hypercyclic vectors for $T$. In other words, $x$ is a hypercyclic vector for $T$ if and only if for every non-empty open subset $U\subset X$, the return set $N(x,U):=\{n\in\N:T^nx\in U\}$ is non-empty (or equivalently infinite). 

The first appearance of such an operator in the literature goes back to Birkhoff \cite{Birk} in 1929 who proved that the translation
operators $T_a:f\mapsto f( \cdot +a)$ on the space of entire functions $H(\C)$ are hypercyclic when $a\neq0$. However, even if some other important examples have been given before, the intense study of hypercyclicity  only began thirty years ago with the work of Kitaï \cite{Kitai} in 1982 and the deep work of Godefroy and Shapiro \cite{Godefroyshapiro} in 1991. Moreover, chaotic operators are also introduced in this last paper; an operator $T$ is said to be chaotic if it satisfies the two following conditions:
\begin{itemize}
	\item $T$ is hypercyclic,
	\item $T$ admits a dense set of periodic points.
\end{itemize}
Among all the questions that have been considered in this field, those
concerning the dynamical properties of an operator related to another
one have interested many authors. One may think for example of Ansari's result
\cite{Ansari} stating that every iterate of a hypercyclic operator is
itself hypercyclic, or of the so-called  Herrero's $T\times T$ problem: is $T\times T$ hypercyclic as soon as $T$ is? This question has been solved in the negative by De la Rosa and Read \cite{Delarosaread} in 2009 for some Banach spaces and by Bayart and Matheron \cite{Baymathyppb} for classical Banach spaces including the Hilbert spaces.

Soon after its introduction by Bayart and Grivaux \cite{Baygrihyp} in 2004, frequent hypercyclicity became one of the main subject of studies in linear dynamics. Recall that an operator $T$ is said to be frequently hypercyclic if there exists a vector $x\in X$ such that for every non-empty open subset $U\subset X$, the return set $N(x,U):=\{n\in\N: T^nx\in U\}$ has positive lower density  i.e. $\underline{d}(N(x,U))>0$, where for any $A\subseteq \N$
\[\underline{d}(A)=\liminf_{N\to+\infty}\frac{\#([0,N]\cap A)}{N+1}.\]
The choice of ``measuring''  the size of the return set thanks to the
lower natural density was motivated by the notion of ergodic operator
and Birkhoff's Ergodic Theorem. However some other notions arose
naturally later while varying the densities. In 2009, Shkarin
\cite{Shkaspectrumfhc} 
introduced the weaker notion of
$\mathcal{U}$-frequent hypercyclicity by replacing the lower density
by the upper natural density $\overline{d}$ defined for every
$A\subseteq\N$ by
\[\overline{d}(A)=\limsup_{N\to+\infty}\frac{\#([0,N]\cap A)}{N+1}.\]
In 2015, Bès, Menet, Peris and Puig \cite{Besmenperpuig} obtained an even weaker notion, called reiterative hypercyclicity, by replacing these densities by the upper Banach density $\overline{Bd}$ which is defined by 
\[\overline{Bd}(A)=\lim_{N\to+\infty}\frac{b_N}{N}\quad \text{where}
  \quad b_N=\limsup_{k\to+\infty}\#(A\cap [k+1,k+N]).\] 
Let us finally mention that in 2017, Ernst and Mouze
\cite{ErnstmouzequantFHC} used the so-called weighted densities to
provide a family of dynamical notions indexed by positive sequences
$a:=(a_n)_{n\in\N}$ satisfying some growth conditions. We will
introduce precisely these notions of weighted frequent hypercyclicity
in Section~\ref{sec:1}. 

 Herrero's $T\times T$ problem also spawned several similar problems. Indeed, Grosse-Erdmann and Peris proved
\cite{Gropefhdense} that if $T$ is frequently hypercyclic then $T$ is
weakly-mixing i.e. $T\times T$  is hypercyclic. This result has been
improved by Bes, Menet, Peris and Puig \cite{Besmenperpuig} in 2015
who proved that reiterative hypercyclicity also implies that $T\times T$ is hypercyclic. One can wonder if $T\times T$ can satisfy stronger dynamical properties. For instance, an important open problem in linear dynamics posed by Bayart and Grivaux \cite{Baygrifrequentlyhcop} in 2006 consists in determining if $T\times T$ is frequently hypercyclic as soon as $T$ is frequently hypercyclic. This question is also open for the notions of reiterative hypercyclicity and $\mathcal{U}$-frequent hypercyclicity:

\begin{question}\label{Q1}
	If $T$ is $\mathcal{U}$-frequently hypercyclic (resp. reiteratively hypercyclic), does it automatically imply that $T\times T$ is $\mathcal{U}$-frequently hypercyclic (resp. reiteratively hypercyclic)?
\end{question}

A study of weighted densities giving rise to dynamical notions lying between $\mathcal{U}$-frequent hypercyclicity and frequent hypercyclicity was recently conducted in \cite{Menet2019}. In this article, we will work with weighted densities giving rise to dynamical notions lying between reiterative hypercyclicity and $\mathcal{U}$-frequent hypercyclicity. The considered weight sequences will be chosen in a particular class $\mathcal{S}$ and will create a bridge between reiterative hypercyclicity and $\mathcal{U}$-frequent hypercyclicity. 

Recently there has been an increasing interest toward to the comparison of chaos with the other
dynamical notions obtained thanks to different densities. In particular, Menet \cite{MenetChaos} proved in 2017 that chaos does not
imply frequent hypercyclicity nor $\mathcal{U}$-frequent hypercyclicity but implies reiterative hypercyclicity. However, since
$\mathcal{U}$-frequent hypercyclicity with respect to weights in $\mathcal{S}$ lies between these last two notions, the question is still open in this case.

\begin{question}\label{Q2}
	Is every chaotic operator $\mathcal{U}$-frequently hypercyclic with respect to some weight in $\mathcal{S}$?
\end{question}

The aim of this article is to answer both Question \ref{Q1} and \ref{Q2}. 

In Section 1, we give a precise definition of $\mathcal{U}$-frequent
hypercyclicity with respect to weights and we focus on the links
between this notion and the pre-existing ones: $\mathcal{U}$-frequent
hypercyclicity and reiterative hypercyclicity. In fact, we prove that
if we restrict to some natural class $\SS$ of weights, $\mathcal{U}$-frequent
hypercyclicity implies $\mathcal{U}$-frequent hypercyclicity with respect to
weights in $\SS$ which implies reiterative hypercyclicity. More
precisely, we get the following result: 
\begin{theorem*}
	Let $T$ be a bounded operator on a Banach space $X$. Then:
	\begin{enumerate}
		\item $T$ is $\mathcal{U}$-frequently hypercyclic if and only if $T$ is $\mathcal{U}$-frequently hypercyclic with respect to each $a\in\SS$.
		\item $T$ is reiteratively hypercyclic if and only if $T$ is $\mathcal{U}$-frequently hypercyclic with respect to some $a\in\SS$.
	\end{enumerate}
	Moreover, 
	\[UFHC(T)=\cap_{a\in\SS}UFHC_a(T)\text{ and }RHC(T)=\cup_{a\in\SS}UFHC_a(T).\]
\end{theorem*}

In Section 2, we are interested in the properties that $T\times T$ can inherit from $T$. We  obtain a result which applies to $\AA$-hypercyclicity for some upper Furstenberg family $\AA$ using the recent work of Bonilla and Grosse-Erdmann \cite{BongrossUFHC}. In particular, we  get the following result which gives a positive answer to Question \ref{Q1}.
\begin{theorem*}Let $T$ be a bounded operator on a Banach space $X$ and $n\ge 1$.
\begin{itemize}
\item If $T$ is $\mathcal{U}$-frequently hypercyclic then the $n$-fold product $T\times\cdots\times T$ is $\mathcal{U}$-frequently hypercyclic.
\item  If $T$ is reiteratively hypercyclic then the $n$-fold product $T\times\cdots\times T$ is reiteratively hypercyclic.
\end{itemize}
\end{theorem*}

Finally, in Section 3, using the previous work of Grivaux, Matheron
and Menet \cite{Monster} on operators of C-type, we give an example, for
any fixed $a\in\SS$, of a chaotic operator on $\ell^1(\N)$ which is not
$\mathcal{U}$-frequently hypercyclic for $a$, answering Question~\ref{Q2}.

\section{$\mathcal{U}$-frequent hypercyclicity with respect to $a\in\mathcal{S}$}\label{sec:1}

\begin{notation}
	Let $a:=(a_n)_{n\in\N}$ be a non-decreasing sequence of positive numbers. We define 
	\[v_{n}(a):=\frac{a_n}{\sum_{j=0}^{n-1}a_j}.\]
	In the following, we will focus on a specific class of weighted densities defined with sequences satisfying some growth conditions. We use the following notations for these sequences:
	\begin{align*}
	\SS=\{(a_n)_{n\in\N}\in \R_{+}^{\N}: (a_n)_{n\in\N}\text{ is non-decreasing, } a_n\to+\infty,&\\ (v_{n}(a))_{n\in\N}\text{ is non-increasing and }v_{n}(a)\to 0&\}.\\
	\end{align*}
\end{notation}
For instance, the sequence $a_n=n^\alpha$ belongs to $\SS$ for every $\alpha>0$ while the sequence $a_n=\alpha^n$ is not in $\SS$ for any $\alpha>1$. The work of Freedman and Sember \cite{Freedman} on weighted densities allows to give the following definition.

\begin{definition}
	Let $a:=(a_n)_{n\in\N}$ be a positive sequence such that $\sum_{n=1}^{+\infty}a_n=+\infty$ and $A\subseteq\N$. Then, the upper $a$-density of $A$ is given by
	\[\overline{d}_a(A)=\limsup_{N\to+\infty} \frac{\sum_{j=0}^{N}a_j\indi{A}(j)}{\sum_{j=0}^{N}a_j}.\]
\end{definition}
	 
	 These objects are densities in the sense of Freedman and Sember and enjoy all the classical properties of densities. Moreover, they give rise to dynamical notions.
\begin{definition}
	 Let $a:=(a_n)_{n\in\N}$ be a positive sequence such that
         $\sum_{n=1}^{+\infty}a_n=+\infty$ and let also $T$ be a
         continuous linear operator on a Banach space $X$. Then, $T$
         is said to be $\mathcal{U}$-frequently hypercyclic with
         respect to $a$ (in short $UFHC_{a}$), if there exists $x\in X$ such that for every non-empty open subset $U\subset X$,
	 \[\overline{d}_a(N(x,U))>0.\]
	 In that case, the vector $x$ is also said to be $\mathcal{U}$-frequently hypercyclic with respect to $a$ and we denote by $UFHC_a(T)$ the set of such vectors.	 
\end{definition}

We recall that the family $\SS$ is intended to give rise to dynamical notions between reiterative hypercyclicity and $\mathcal{U}$-frequent hypercyclicity. In \cite{ErnstmouzequantFHC,ErnstmouzeFUCdensities}, the authors focused on the links between frequent hypercyclicity and such dynamical notions given by lower densities. Moreover, the implication between $\mathcal{U}$-frequent hypercyclicity with respect to $a\in \SS$ and $\mathcal{U}$-frequent hypercyclicity is given to us by the following lemma.

\begin{lemma}[{\cite{ErnstmouzequantFHC}}]\label{lemErnst} Let $a=(a_n)_{n\in\N}$ and $b=(b_n)_{n\in\N}$ be positive sequences such that $\sum_{n\in\N} a_n=\sum_{n\in\N} b_n=+\infty.$ Assume that the sequence $\left(\frac{a_n}{b_n}\right)_{n\in\N}$ is eventually decreasing to zero. 
	Then, for every subset $A\subseteq\mathbb{N},$ we have 
	\[\underline{d}_{b}(A)\leq \underline{d}_{a}(A) \leq \overline{d}_{a}(A) \leq \overline{d}_{b}(A).\]
\end{lemma}

The assumption $a_n\nearrow+\infty$ in the definition of $\SS$ ensures that $\overline{d}\leq\overline{d}_a$ for every $a\in\SS$ and thus that $\mathcal{U}$-frequent hypercyclicity implies $\mathcal{U}$-frequent hypercyclicity with respect to $a$. Moreover, the assumption $v_{n}(a)\searrow 0$ ensures that we are not dealing with the case where the upper-density associated to $a$ is supported by some infinite sequence i.e. we exclude the cases where there exists a sequence $(n_k)_{k\in\N}$ so that $\overline{d}_a(I)>0$ as soon as $\#(I\cap\{n_k\}_{k\in\N})=\infty$.

One can wonder if there is a gap between $\mathcal{U}$-frequent hypercyclicity and $\mathcal{U}$-frequent hypercyclicity with respect to $a\in \SS$. In fact, this is not the case since there is no difference between $\mathcal{U}$-frequent hypercyclicity and $\mathcal{U}$-frequent hypercyclicity with respect to some $a$ having polynomial growth.

\begin{proposition}\label{Prop1}
	For every $\alpha>0$, the sequence $a=
        (n^{\alpha})_{n \in \N} \in \SS$ satisfies % there exists $a\in\SS$ such that for every $I\subseteq\N$:
	\[\frac{1}{1+\alpha}\overline{d}_a(I)\leq\overline{d}(I)\leq
          \overline{d}_a(I).\]
for every $I \subseteq \N$.
\end{proposition}

\begin{proof}
Let $I$ be a subset of $\mathbb{N}$. If $I$ is finite, the result is obvious since $\overline{d}_a(I)=\overline{d}(I)=0$. We can thus assume that $I$ is infinite and we denote by $(n_k)_{k\in\N}$ the increasing enumeration of $I$. Therefore, we have
	\begin{align*}
	\overline{d}_a(I)&=\limsup_{k\to+\infty}\frac{\sum_{j=0}^{k}a_{n_j}}{\sum_{j=0}^{n_k}a_j}\\
	&\leq \limsup_{k\to+\infty}\frac{(k+1)n_{k}^{\alpha}}{\sum_{j=0}^{n_k}a_j}\\
	&\leq \limsup_{k\to+\infty}(\alpha+1)\frac{(k+1)n_{k}^{\alpha}}{n_k^{\alpha+1}}\\
	&=(\alpha+1)\limsup_{k\to+\infty}\frac{(k+1)}{n_k}\\
	&=(\alpha+1)\overline{d}(I).
	\end{align*}
The second inequality follows from Lemma~\ref{lemErnst}.
\end{proof}

The next result allows to order the upper Banach density and the weighted upper densities $\overline{d}_a$ with $a\in \SS$.

\begin{proposition}\label{Prop2}
	For every $a\in\SS$ and every $I\subseteq\N$, $\overline{d}_a(I)\leq \overline{Bd}(I)$.
\end{proposition}

\begin{proof}
	Let $I\subseteq\N$ and $\delta>0$ be any real number satisfying $\overline{Bd}(I)<\delta$.
	By definition of the upper Banach density, there exists $p\geq1$ and $k_0\in\N$ such that for every $k\geq k_0$,
	\[\#(I\cap[k,k+p[)\leq\delta p.\]
	Consider any integer $m_0$ satisfying $m_0p\geq k_0$. Of
        course, the upper $a$-density of a set remains unchanged if
        one removes a finite number of elements. Then, we obtain
	\begin{align*}
	\overline{d}_a(I)&=\limsup_{n\to+\infty} \frac{\sum_{j=m_0p}^{n}a_j\indi{I}(j)}{\sum_{j=m_0p}^{n}a_j}\\
	&\leq \limsup_{m\to+\infty}\frac{\sum_{s=m_0}^{m}\sum_{j=sp}^{(s+1)p-1}a_j\indi{I}(j)}{\sum_{s=m_0}^{m-1}\sum_{j=sp}^{(s+1)p-1}a_j}\\
	&\leq\limsup_{m\to+\infty}\frac{\sum_{s=m_0}^{m}a_{(s+1)p-1}\sum_{j=sp}^{(s+1)p-1}\indi{I}(j)}{\sum_{j=m_0p}^{mp-1}a_j}\\
	&\leq \limsup_{m\to+\infty} \frac{\delta\sum_{s=m_0}^{m}pa_{(s+1)p-1}}{\sum_{j=m_0p}^{mp-1}a_j}.
	\end{align*}
	In order to avoid heavy computations, we know focus on the term lying in the sum above and we remark that by definition of $v_n(a)$, we have the following relations:
	\[a_n=v_n(a)\Big(\sum_{j=0}^{n-1}a_j\Big) \quad
          \text{and}\quad
          \Big(\sum_{j=0}^{n-1}a_j\Big)(1+v_n(a))=\sum_{j=0}^n a_j \, .\]
	 Therefore, keeping in mind that the sequence $(v_n(a))_{n\in\N}$ is decreasing and that $m_0\leq s\leq m$, we get
	\begin{align*}
	pa_{(s+1)p-1}&=
                       \sum_{l=0}^{p-1}a_{sp+l}\frac{a_{(s+1)p-1}}{a_{sp+l}}\\
& =   \sum_{l=0}^{p-1}a_{sp+l}\frac{v_{(s+1)p-1}(a)\sum_{j=0}^{(s+1)p-2}a_{j}}{a_{sp+l}}\\
	&=\sum_{l=0}^{p-1}a_{sp+l}\frac{v_{(s+1)p-1}(a)\left(\prod_{j=sp}^{(s+1)p-2}(1+v_j(a))\right)\sum_{j=0}^{sp-1}a_j}{a_{sp+l}}\\
	&\leq \sum_{l=0}^{p-1}a_{sp+l} v_{sp}(a)(1+v_{sp}(a))^{p-1} \frac{\sum_{j=0}^{sp-1}a_j}{a_{sp}}\\
	&=\sum_{l=0}^{p-1}a_{sp+l}(1+v_{sp}(a))^{p-1}\\
	&\leq (1+v_{m_0p}(a))^{p-1}\sum_{l=0}^{p-1}a_{sp+l}.
	\end{align*}
We deduce from this upper bound together with the previous inequality that
	\begin{align*}
	\overline{d}_a(I)&\leq \limsup_{m\to+\infty}\frac{\delta(1+v_{m_0p}(a))^{p-1}\sum_{j=m_0p}^{mp+p-1}a_j}{\sum_{j=m_0p}^{mp-1}a_j}\\
	&= \delta(1+v_{m_0p}(a))^{p-1}\left(1+\limsup_{m\to+\infty}\frac{\sum_{l=0}^{p-1}a_{mp+l}}{\sum_{j=m_0p}^{mp-1}a_j}\right)\\
	% &=\delta(1+v_{m_0p}(a))^{p-1}\left(1+\limsup_{m\to+\infty}\sum_{l=0}^{p-1}v_{mp+l}(a)\left(\frac{\sum_{j=0}^{mp-1}a_j}{\sum_{j=m_0p}^{mp-1}a_j}+\sum_{s=0}^{l-1}v_{mp+s}(a)\prod_{j=0}^{s-1}(1+v_{mp+j}(a))\right)\right)\\
	% &=\delta(1+v_{m_0p}(a))^{p-1}\left(1+\limsup_{m\to+\infty}pv_{mp}(a)\left(\frac{\sum_{j=0}^{mp-1}a_j}{\sum_{j=m_0p}^{mp-1}a_j}+pv_{mp}(a)(1+v_{mp}(a))^p\right)\right).
	\end{align*}
where
	\begin{align*}
\limsup_{m\to+\infty}\frac{\sum_{l=0}^{p-1}a_{mp+l}}{\sum_{j=m_0p}^{mp-1}a_j}&=\limsup_{m\to+\infty}\frac{\sum_{l=0}^{p-1}a_{mp+l}}{\sum_{j=0}^{mp-1}a_j}\\
	&=\limsup_{m\to+\infty}\sum_{l=0}^{p-1}v_{mp+l}(a)\left(1+\sum_{s=0}^{l-1}v_{mp+s}(a)\prod_{j=0}^{s-1}(1+v_{mp+j}(a))\right)\\
	&\leq \limsup_{m\to+\infty}pv_{mp}(a)\left(1+pv_{mp}(a)(1+v_{mp}(a))^p\right)=0
	\end{align*}
	since $v_{mp}(a)\underset{m\to\infty}{\longrightarrow}0$ and $p$ is fixed from the beginning. We then obtain
	\[\overline{d}_a(I)\leq \delta(1+v_{m_0p}(a))^{p-1}.\]
	Moreover, this last inequality holds for any $m_0$ satisfying $m_0p\geq k_0$, hence making $m_0$ go to infinity, this yields $\overline{d}_a(I)\leq \delta$ where $\delta$ was any real number bigger than $\overline{Bd}(I)$. Hence, \[\overline{d}_a(I)\leq \overline{Bd}(I).\]
\end{proof}

\begin{figure}[!h]
	\centering

	\includegraphics[width=14cm]{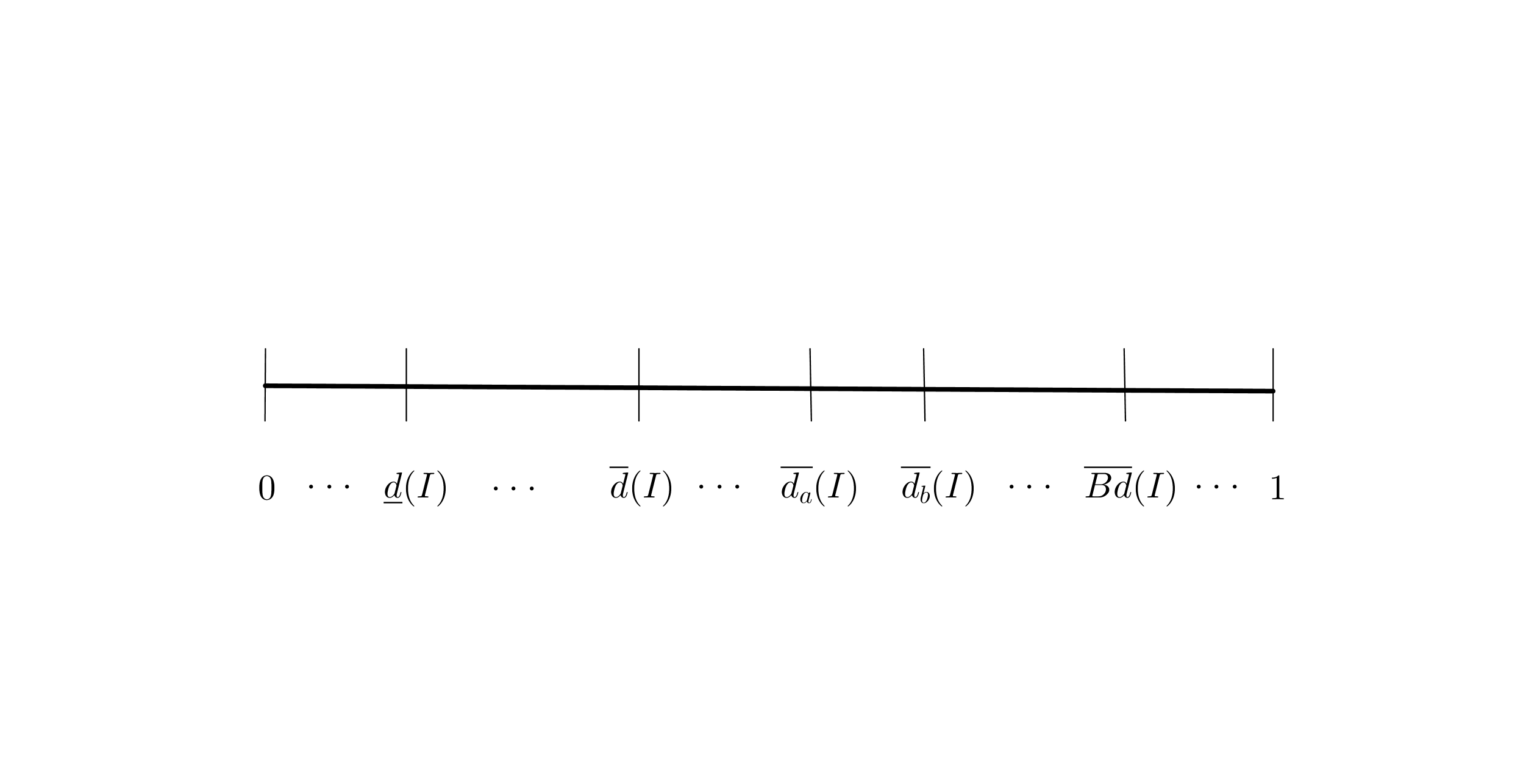}
	\caption{Ordering of densities when $\frac{a_n}{b_n}\searrow0$ and $a,b\in \SS$} 
	\label{Im1}
\end{figure}

As it is summarized in Figure \ref{Im1}, we have a natural ordering
between our weighted upper densities and the upper Banach density. In
addition, we can show that upper densities $\overline{d}_{a}$ with $a\in \SS$ allow to approach the upper Banach density.

\begin{proposition}\label{Prop3}
	Let $(A_n)_{n\in\N}$ be a sequence of sets of non-negative integers.
	Then, there exists $a\in\SS$ such that for every $n\in\N$
	\[\overline{Bd}(A_n)\leq e\cdot\overline{d}_a(A_n),\]
	where $e$ is Euler's number.
\end{proposition}

\begin{proof}
	The inequality is trivially satisfied for integers $n$ so that
        $\overline{Bd}(A_n)=0$. Consequently, in what follows we will
        always assume that $\overline{Bd}(A_n)>0$. For every $n\in\N$,
        let $0<\delta_n<\overline{Bd}(A_n)$ and let $(I_n)_{n\in\N}$
        be a partition of $\N$ with $\# I_n=\infty$ for every $n \in \N$. Then, by definition of the upper Banach density, there exists an increasing sequence of integers $(k_p)_{p\in\N}$ such that
	\begin{enumerate}
		\item $k_{p+1}>k_p+p$ for every $p\in\N$,
		\item $\#\left(A_n\cap[k_p,k_p+p[\right)\geq p\cdot\delta_n$ if $p\in I_n$.
	\end{enumerate}
	We define our weight sequence $a$ by setting
	\[a_0=1\quad
          \text{and}\quad a_j=\frac{1}{p}\sum_{l=0}^{j-1}a_l\,\,\text{
            when }\, \, j\in[k_p,k_{p+1}[\]
	and we claim that this sequence has the properties we are looking for.
	We begin by proving that $a\in\SS$.
	Indeed,
	\[a_{k_p}=\frac{1}{p}\sum_{l=0}^{k_p-1}a_l=\frac{1}{p}\left(a_{k_p-1}+(p-1)a_{k_p-1}\right)=a_{k_p-1}.\]
	Moreover, if $k\in\, ]k_p,k_{p+1}[$ then
	\[a_k=\frac{1}{p}\sum_{l=0}^{k-1}a_l=\frac{1}{p}\left(a_{k-1}+pa_{k-1}\right)=\frac{p+1}{p}a_{k-1}\geq a_{k-1}.\]
	Thus, $a$ is non-decreasing. In addition, we have
	% \[a_{k_p}=\left(\frac{p+1}{p}\right)^{k_{p}-k_{p-1}}a_{k_{p-1}}\geq\frac{p+1}{p}a_{k_{p-1}}\geq
        %   \frac{p+1}{2}a_{k_1}.\]
	\[
a_{k_p}=a_{k_{p}-1} = \left(\frac{p}{p-1}\right)^{k_{p}-k_{p-1}}a_{k_{p-1}}\geq\frac{p}{p-1}a_{k_{p-1}}
\]
hence  by induction
$
a_{k_p} \geq p a_{k_{1}}
$
and it follows that
$a_{n}\underset{n\to\infty}{\longrightarrow}+\infty$. 
	Furthermore, by definition $v_{k}(a)=\frac{1}{p}$ if $k\in[k_p,k_{p+1}[$ and thus $(v_n(a))_{n\in\N}$ is non-increasing and tends to zero which proves that $a\in\SS$.
	Since, for every $n\in\N$, $\delta_n$ is any positive real
        number smaller than $\overline{Bd}(A_n)$, it only remains to
        prove that for every $n\in\N$, the inequality
        $\overline{d}_a(A_n)\geq \frac{\delta_n}{e}$ holds. One has
	\begin{align*}
	\overline{d}_a(A_n)&\geq \limsup_{p\in I_n} \frac{\sum_{l=0}^{k_p+p-1}a_l\indi{A_n}(l)}{\sum_{l=0}^{k_p+p-1}a_l}\\
	&\geq \limsup_{p\in I_n} \frac{a_{k_p}\sum_{l=k_p}^{k_p+p-1}\indi{A_l}(l)}{\sum_{l=0}^{k_p+p-1}a_l}\\
	&\geq \limsup_{p\in I_n} \frac{a_{k_p}\delta_n p}{\sum_{l=0}^{k_p+p-1}a_l}\\
	&= \limsup_{p\in I_n} \delta_n\frac{\sum_{l=0}^{k_p-1}a_l}{\sum_{l=0}^{k_p+p-1}a_l}\\
	&\geq\delta_n \limsup_{p\in I_n}\frac{\sum_{l=0}^{k_p-1}a_l}{\left(\frac{p+1}{p}\right)^p\sum_{l=0}^{k_p-1}a_l}\\
	&=\frac{\delta_n}{e}
	\end{align*}
hence the conclusion. 
\end{proof}

Of course, the inequality $\overline{Bd}(A)\leq e\cdot\overline{d}_a(A)$ is not satisfied for every set $A$.

\begin{proposition}\label{propnk}
For every $a\in \SS$, there exists an infinite set $A\subseteq \N$ such that $\overline{Bd}(A)=1$ and $\overline{d}_a(A)=0$.
\end{proposition}
\begin{proof}
Let $a\in \SS$. Since for every increasing sequence $(n_k)_{k\in\N}$, $\overline{Bd}(\cup_{k\in\N}[n_k,n_k+k])=1$, it suffices to show that there exists an increasing sequence $(n_k)_{k\in\N}$ such that $\overline{d}_a(\cup_{k\in\N}[n_k, n_k+k])=0$ . We remark that for any
        increasing sequence $(n_k)_{k\in\N}$, we have for every $n \in [n_{k}, n_{k}+k]$
	% \begin{align*}
	% \overline{d}_a(\cup_{k\in\N}[n_k;n_k+k])&\leq\limsup_{k\to+\infty}\frac{\sum_{j=1}^{k+1}\sum_{l=0}^{j}a_{n_j+l}}{\sum_{l=0}^{n_{k+1}+k+1}a_l}\\
	% &\leq\limsup_{k\to+\infty}\frac{\sum_{l=0}^{n_k+k}a_l}{\sum_{l=0}^{n_{k+1}}a_l}+\frac{\sum_{l=0}^{k+1}a_{n_{k+1}+l}}{\sum_{l=0}^{n_{k+1}+k+1}a_l}\\
	% &:=\limsup_{k\to+\infty}\ (I)+(II).
	% \end{align*}
	\begin{align*}
\frac{\sum_{l=0}^{n}a_l\indi{\cup_{k\in\N}[n_k,n_k+k]}(l)}{\sum_{l=0}^{n}a_l}
&=\frac{\sum_{j=1}^{k-1}\sum_{l=0}^{j}a_{n_{j}+l}+ \sum_{l=n_{k}}^{n}a_{l}}{\sum_{l=0}^{n}a_l}\\
& \leq
  \frac{\sum_{j=1}^{k-1}\sum_{l=0}^{j}a_{n_{j}+l}}{\sum_{l=0}^{n_{k}}a_l}+\frac{\sum_{l=n_{k}}^{n}a_{l}}{\sum_{l=0}^{n}a_l}
          \\
&:= (I)+(II).
	\end{align*} 

	Thus, to prove our result it suffices to show that for every choice having already been made for the $k-1$ first terms of the sequence $(n_k)_{k\in\N}$, we are able to choose $n_{k}$ so that both $(I)$ and $(II)$ are arbitrarily small. 
	First, as $(a_n)_{n\in\N}$ belongs to $\SS$, it grows to infinity and then the term $(I)$ above can be made arbitrarily small by choosing a big enough $n_{k}$. For the second term $(II)$, remark that since $(a_n)_{n\in\N}$  is increasing and $(v_n(a))_{n\in\N}$ is decreasing we obtain
	% \begin{align*}
	% \frac{\sum_{l=0}^{k+1}a_{n_{k+1}+l}}{\sum_{l=0}^{n_{k+1}+k+1}a_l}&\leq (k+2)\frac{a_{n_{k+1}+k+1}}{a_{n_{k+1}+k+1}+\sum_{l=0}^{n_{k+1}+k}a_l}\\
	% &=\frac{k+2}{1+(v_{n_{k+1}+k+1}(a))^{-1}}\\
	% &\leq\frac{k+2}{1+(v_{n_{k+1}}(a))^{-1}}.
	% \end{align*}
	\begin{align*}
	\frac{\sum_{l=n_{k}}^{n}a_{l}}{\sum_{l=0}^{n}a_l}&\leq (k+1)\frac{a_{n}}{a_{n}+\sum_{l=0}^{n-1}a_l}\\
	&=\frac{k+1}{1+(v_{n}(a))^{-1}}\\
	&\leq\frac{k+1}{1+(v_{n_{k}}(a))^{-1}}.
	\end{align*}
	
	Finally, this last expression can be made arbitrarily small by choosing $n_{k}$ big enough because $(v_n(a))_{n\in\N}$ decreases to zero.
\end{proof}

 We are now able to prove the main theorem of this section concerning the links between the notions of $\mathcal{U}$-frequent hypercyclicity,  $\mathcal{U}$-frequent hypercyclicity with respect to $a\in \SS$ and reiterative hypercyclicity. 

\begin{theorem}\label{TH_eq_UFHCa_RHCa}
	Let $T$ be a bounded operator on a Banach space $X$. Then:
\begin{enumerate}
	\item $T$ is $\mathcal{U}$-frequently hypercyclic if and only if $T$ is $\mathcal{U}$-frequently hypercyclic with respect to each $a\in\SS$.\label{eqTHUa1}
	\item $T$ is reiteratively hypercyclic if and only if $T$ is $\mathcal{U}$-frequently hypercyclic with respect to some $a\in\SS$.\label{eqTHUa2}
\end{enumerate}
Moreover, 
\[UFHC(T)=\cap_{a\in\SS}UFHC_a(T)\text{ and }RHC(T)=\cup_{a\in\SS}UFHC_a(T).\]
\end{theorem}
\begin{proof}
	We begin by the first statement (\ref{eqTHUa1}) of the theorem. The sufficiency is already known since $\overline{d}\leq \overline{d}_a$ for every $a\in\SS$ by Lemma~\ref{lemErnst} and in addition $UFHC(T)\subseteq\cap_{a\in\SS}UFHC_a(T)$.
	
	The necessary part is a consequence of Proposition \ref{Prop1} as there exists $b\in\SS$ such that both $\mathcal{U}$-frequent hypercyclicity and $\mathcal{U}$-frequent hypercyclicity with respect to $b$ are equivalent. In particular, if $T$ is $\mathcal{U}$-frequently hypercyclic for every $a\in\SS$ then $T$ is $\mathcal{U}$-frequently hypercyclic and $\cap_{a\in\SS}UFHC_a(T)\subseteq UFHC(T)$.
	
	For the statement (\ref{eqTHUa2}) of the theorem, the necessity is a direct consequence of Proposition \ref{Prop2} and $\cup_{a\in\SS} UFHC_a(T)\subseteq RHC(T)$.
	Moreover, by Proposition \ref{Prop3} if $x\in RHC(T)$ and if $(U_n)_{n\in\N}$ is a basis of open sets from $X$, then there exists $a\in\SS$ such that for every $n\in \N$,
	\[\overline{d}_a(N(x,U_n))\geq\frac{1}{e}\overline{Bd}(N(x,U_n))>0.\]
	Thus, $x$ is $\mathcal{U}$-frequently hypercyclic with respect to this particular $a$, hence 
	\[RHC(T)\subseteq \cup_{a\in\SS}UFHC_a(T)\]
	which proves the sufficient part and the theorem.
\end{proof}
\begin{remark}
Note that if $T$ is $\mathcal{U}$-frequently hypercyclic, we even have that $UFHC(T)=UFHC_{a}(T)$ for some $a\in \SS$.
\end{remark}

 In 2015, Bès, Menet, Peris and Puig \cite{Besmenperpuig} proved that if $T$ is reiteratively hypercyclic then
$RHC(T)=HC(T)$. It is worth noting that this result together with Theorem~\ref{TH_eq_UFHCa_RHCa}
 yields  the next corollary.

\begin{corollary}
	If $T$ is a reiteratively hypercyclic operator then
	\[HC(T)=\cup_{a\in\SS}UFHC_a(T).\]
\end{corollary}

Note also that for every $a\in \SS$, each set $UFHC_a(T)$ is either empty or comeager~\cite{BongrossUFHC}. However, we do not have in general that $UFHC_a(T)=HC(T)$.

\begin{proposition}
Let $B$ denote the backward shift on $\ell^{p}$, $1 \leq p < +
\infty$. 
	For every $a\in\SS$, $2B$ is $\mathcal{U}$-frequently hypercyclic with respect to $a$ but $UFHC_a(2B)\neq HC(2B)$.
\end{proposition}
\begin{proof}
Let $a\in\SS$. We know thanks to Bayart and Grivaux \cite{Baygrihyp} that $2B$ is frequently hypercyclic and thus $\mathcal{U}$-frequently hypercyclic with respect to $a$. Moreover, by Proposition~\ref{propnk}, there exists an increasing sequence $(n_k)$ such that $\overline{d_a}(	\cup_{k\in\N}[n_k, n_k+k])=0$ and it is easy to see that we can construct a hypercyclic vector $x$ supported by  $\cup_{k\in\N}[n_k, n_k+k]$ i.e. which can be written as $x=\sum_{j\in\cup_{k\in\N}[n_k, n_k+k]}x_j e_j$.
	This vector, though, cannot be $\mathcal{U}$-frequently
        hypercyclic with respect to $a$ because
        $N(x,B(e_0,\frac{1}{2}))\subseteq\{j\in\N:\ 2^j\vert
        x_j\vert\geq \frac{1}{2}\}\subseteq \cup_{k\in\N}[n_k,n_k+k]$
        and this last set has zero upper $a$-density.	
\end{proof}

\section{Properties of the $n$-fold product of $\mathcal{A}$-hypercyclic operators}

This section aims at obtaining dynamical properties of the $n$-fold product of an
$\mathcal{A}$-hypercyclic operator, where $\mathcal{A}$ is an upper
Furstenberg family. Let us start by recalling the definition of this last notion.  

\begin{definition}\cite{BongrossUFHC}
\	A non-empty family $\AA$ of subsets of $\N$ is called a Furstenberg family $\AA$ if 
\begin{enumerate}
		\item $\emptyset\notin\AA$,
		\item $A\in\AA,\ A\subset B\ \implies\ B\in\AA$.
		\end{enumerate}
If, in addition,  $\AA$ can be written as $\AA=\cup_{\delta\in D}\AA_{\delta}$ with $\AA_{\delta}=\cap_{\mu\in M}\AA_{\delta,\mu}$,
		for some families $\AA_{\delta,\mu}$ ($\delta\in D,\ \mu\in M$), where $D$ is arbitrary and $M$ is countable and satisfies:
		\begin{enumerate}[(i)]
			\item for any $A\in\AA_{\delta,\mu}$, there exists a finite set $F\subset \N$ such that
			\[A\cap F\subset B\ \implies\ B\in\AA_{\delta,\mu},\]
			\item for any $A\in\AA$ there is some $\delta\in D$ such that for all $n\geq0$,
			\[A-n\in\AA_{\delta}.\]
		\end{enumerate}
then $\AA$ is said to be an upper Furstenberg family.
\end{definition}

This notion of Furstenberg family allows us to define some dynamical properties.

\begin{definition}\cite{Besmenperpuig}
	Let $\AA$ be a Furstenberg family and let $T$ be a bounded operator on a Banach space $X$. Then, $T$ is called $\AA$-hypercyclic if there exists a vector $x\in X$ such that for any non-empty open set $U\subseteq X$,
	$N(x,U)\in\AA$.
\end{definition}

\begin{example}\label{EX_A}
	This definition permits to recover some already known dynamical notions:
	\begin{enumerate}
		\item If $\AA_{ud}$ is the family of sets of positive
                  upper density, then $\AA_{ud}$-hypercyclicity
                  corresponds to $\mathcal{U}$-frequent
                  hypercyclicity. Moreover,
                  $\AA_{ud}=\cup_{\delta\in\R_{+}^{*}}\cap_{n\in\N}
                  \AA_{\delta,n}$ where $\AA_{\delta,n}:=\{A\subseteq
                  \N: \, \exists N\geq n \text{ such that } \frac{\#(A\cap [0,N])}{N+1}>\delta\}$.
		\item If $\AA_{uBd}$ is the family of sets of upper
                  Banach density, then $\AA_{uBd}$-hypercyclicity
                  corresponds to reiterative hypercyclicity. Moreover,
                  $\AA_{uBd}=\cup_{\delta\in\R_{+}^{*}}\cap_{n\in\N}~\AA_{\delta,n}$ where $\AA_{\delta,n}:=\{A\subseteq
                  \N: \,  \exists m\geq 0 \text{ such that } \frac{\#(A\cap [m,m+n])}{n+1}>\delta\}$.
		\item Let $a\in \SS$. If $\AA_{ud_{a}}$ is the family of
                  sets of positive upper density with respect to $a$,
                  then $\AA_{ud_{a}}$-hypercyclicity corresponds to
                  $\mathcal{U}$-frequent hypercyclicity with respect
                  to $a$. Moreover,
                  $\AA_{ud_{a}}=\cup_{\delta\in\R_{+}^{*}}\cap_{n\in\N}
                  \AA_{\delta,n}$ where $\AA_{\delta,n}:=\{A\subseteq
                  \N: \, \exists N\geq n\text{ such that } \frac{\sum_{j=0}^{N}a_j\indi{A}(j)}{\sum_{j=0}^{N}a_j}>\delta\}$.
	\end{enumerate}
	Each of these families is an upper Furstenberg family.
\end{example}

One of the main features of upper Furstenberg families is that they
give rise to dynamical properties that satisfy an analogue of the
Birkhoff transitivity theorem. This tool is central and will be one of
the main ingredients in the following proofs.

\begin{theorem}[Birkhoff Theorem for upper Furstenberg families \cite{BongrossUFHC}]\label{TH_Birkhoff_UFF}
	Let $T$ be a bounded operator on a Banach space $X$ and $\AA:=\cup_{\delta\in D}\AA_{\delta}$ be an upper Furstenberg family. Then, the following are equivalent:
	\begin{enumerate}
		\item $T$ is $\AA$-hypercyclic;
		\item For any non-empty open subset $V\subseteq X$, there exists $\delta\in D$ such that for any non-empty open subset $U\subseteq X$, there is some $x\in U$ such that
		\[N(x,V)\in\AA_{\delta}.\]
	\end{enumerate}
\end{theorem}

In what follows, we will work with the notion of upper Furstenberg family and obtain, as a corollary, results for $\mathcal{U}$-frequent hypercyclicity with respect to any $a\in\SS$, $\mathcal{U}$-frequent hypercyclicity and reiterative hypercyclicity.
We start by showing that under some conditions on $\AA$, the n-fold product of an $\AA$-hypercyclic operator is still $\AA$-hypercyclic.

\begin{theorem}\label{TH_AHC_PROD}
	Let $\AA$ be an upper Furstenberg family satisfying
	\begin{enumerate}
		\item $\forall \delta'\in D,\ \exists \delta\in D,\ \forall k\geq0$:
		\[A\in \AA_{\delta'}\implies A-k\in\AA_{\delta}.\]\label{Hyp1}
		\item $\forall A\in\AA,\ \exists \delta\in D$:
		\[\{k\geq0:\ A\cap(A-k)\in\AA_{\delta}\}\text{ has bounded gaps.}\] 
	\end{enumerate}
	If $T$ is $\AA$-hypercyclic then $T\times T$ is $\AA$-hypercyclic.
\end{theorem}

\begin{proof}
Let us first remark that the operator $T$ is weakly mixing, i.e
$T\times T$ is hypercyclic. Let $U,V\subseteq X$ be non-empty open
sets and let $x$ be  an $\AA$-hypercyclic vector for $T$. Since $T$ is
$\AA$-hypercyclic and thus hypercyclic, the set $N(U,V)= \{ n \in \N :
U \cap T^{-n}(V) \neq \emptyset\}$ is not empty. We can then consider $n\in N(U,V)$ and let $U_n=U\cap T^{-n}(V)$ which is a non-empty open set. Therefore, $N(x,U_n)\in \AA$ and by definition of $U_n$, we have
\[N(x,U_n)-N(x,U_n)+n\subset N(U,V)\]
because if $n_1,n_2\in N(x,U_n)$ then $T^{n_2}x\in U$ and $T^{n_1-n_2+n}(T^{n_2}x)=T^{n_1+n}x\in V$.
Let $A=N(x,U_n)$. By assumption, $A-A=\{k\geq 0:\ A\cap(A-k)\ne \emptyset\}$ has bounded gaps and thus $N(U,V)$ has also bounded gaps. It follows from \cite{Gropeweakmix} that $T$ is weakly mixing.

Let $V_1,V_2\subseteq X$ be any non-empty open sets. We will now apply Theorem \ref{TH_Birkhoff_UFF} in order to prove that $T\times T$ is $\AA$-hypercyclic. In other words, we want to find some $\delta\in D$ such that for any non-empty open subsets $U_1,U_2\subseteq X$, there exist $x_1\in U_1$ and $x_2\in U_2$ such that $N(x_1,V_1)\cap N(x_2,V_2)\in\AA_{\delta}$.
	
	Since $T$ is hypercyclic by hypothesis, there exists $n_0\in\N$ such that ${T^{n_0}(V_1)\cap V_2\neq\emptyset}$. Thus, by continuity of $T$, we can find a non-empty open subset $\tilde{V_1}\subseteq V_1$ so that
	\[T^{n_0}(\tilde{V_1})\subseteq V_2.\]
	
	Now, let $y$ be a $\AA$-hypercyclic vector for $T$. Then, by hypothesis, there exists $\delta'\in D$ such that
	\[\KK:=\big\{k\geq0:\ N(y,\tilde{V_1})\cap\left(N(y,\tilde{V_1})-k\right)\in\AA_{\delta'}\big\}\text{ has bounded gaps}.\]
	Since we have fixed $\delta'\in D$, let us define now
        $\delta\in D$ thanks to hypothesis (\ref{Hyp1}). Moreover, since $T$ is weakly-mixing, we know that
        $N(U_1,U_2)$ contains arbitrarily long intervals and since $\KK$
        has bounded gaps, it follows that  there exists $k\in N(U_1,U_2)\cap(\KK+n_0)$, which writes
	\[T^k(U_1)\cap U_2\neq\emptyset\text{ and }N(y,\tilde{V_1})\cap \left(N(y,\tilde{V_1})-(k-n_0)\right)\in\AA_{\delta'}.\]
	There again the continuity of $T$ gives the existence of a non-empty open subset $\tilde{U_1}\subseteq U_1$ such that $T^k(\tilde{U_1})\subseteq U_2$ and we consider $n_1\in N(y,\tilde{U_1})$.
	We claim now that if we let $x_1=T^{n_1}y$ and
        $x_2=T^{n_1+k}y$, then we can prove the desired
        property. Indeed, first since $n_1\in N(y,\tilde{U_1})$, then
        $x_1=T^{n_1}y\in\tilde{U_1}\subseteq U_1$ and since $T^{k}\tilde{U_1}\subseteq U_2$, $x_2=T^{k}x_{1}\in U_2$. Remark also that 
	\[N(x_1,V_1)=N(y,V_1)-n_1\text{ and }N(x_2,V_2)=N(y,V_2)-n_1-k.\] Furthermore, we can  derive $n_0+N(y,\tilde{V_1})\subseteq N(y,V_2)$, hence we obtain
	\begin{align*}
		N(x_1,V_1)\cap N(x_2,V_2)&=\left(N(y,V_1)-n_1\right)\cap\left(N(y,V_2)-n_1-k\right)\\
		&\supseteq \left(N(y,\tilde{V_1})-n_1\right)\cap\left(N(y,\tilde{V_1})+n_0-n_1-k\right)\\
		&=\left(N(y,\tilde{V_1})\cap\left(N(y,\tilde{V_1})-(k-n_0)\right)\right)-n_1.
	\end{align*}
	Thus, by definition of $\delta$, we get that $\left(N(y,\tilde{V_1})\cap\left(N(y,\tilde{V_1})-(k-n_0)\right)\right)-n_1\in\AA_{\delta}$ which yields that $N(x_1,V_1)\cap N(x_2,V_2)$ itself belongs to $\AA_{\delta}$ by definition of upper Furstenberg families.	
\end{proof}

In the previous result, we obtain the $\AA$-hypercyclicity of the 2-fold product $T\times T$. Of course, this is not specific to 2 and can be done for any $n\geq1$ as the following corollary shows. 

\begin{corollary}\label{Cor_AHC_nPROD}
	Let $\AA$ be an upper Furstenberg family satisfying
	\begin{enumerate}
		\item $\forall \delta'\in D,\ \exists \delta\in D,\ \forall k\geq0$:
		\[A\in \AA_{\delta'}\implies A-k\in\AA_{\delta}.\]
		\item $\forall A\in\AA,\ \exists \delta\in D$:
		\[\{k\geq0:\ A\cap(A-k)\in\AA_{\delta}\}\text{ has bounded gaps.}\] 
	\end{enumerate}
If $T$ is $\AA$-hypercyclic, then for every $n\geq1$, the $n$-fold product $\underbrace{T\times\cdots\times T}_{n\text{ times}}$ is $\AA$-hypercyclic. %\underbrace{T\times\cdots\times T}_{n\text{ times}}
\end{corollary}

\begin{proof}
By applying several times Theorem \ref{TH_AHC_PROD}, we get that for every $n\geq 1$, $\underbrace{T\times\cdots\times T}_{2^n\text{ times}}$ is $\AA$-hypercyclic. Finally, $\AA$ being a Furstenberg family, it is in particular hereditarily upward so $\underbrace{T\times\cdots\times T}_{n\text{ times}}$ is $\AA$-hypercyclic for any $n\geq1$.	
\end{proof}

This general result being proved, we can focus on some applications by
fixing the family $\AA$. 
%Let us start by a remark concerning the links
%between the dynamical notions studied in this paper and the
%weakly-mixing property. 
%
%\begin{remark}\label{rem:weaklymixing}
%If $T$ is reiteratively hypercyclic, then by Proposition 25 from
%\cite{Besmenperpuig}, we know that $T$ is  weakly-mixing. Using
%Theorem \ref{TH_eq_UFHCa_RHCa}, we get that  $\mathcal{U}$-frequent
%hypercyclicity with respect to some weight $a\in\SS$ also implies
%weakly-mixing. Finally, we also know that any $\mathcal{U}$-frequently
%hypercyclic operator is automatically weakly-mixing. 
%\end{remark}

 We begin by considering the family of sets of positive upper density.

\begin{corollary}
	If $T$ is $\mathcal{U}$-frequently hypercyclic, then for any $n\geq1$, the $n$-fold product $T\times\cdots\times T$ is $\mathcal{U}$-frequently hypercyclic.
\end{corollary}

\begin{proof}
	It suffices to show that conditions of Corollary \ref{Cor_AHC_nPROD} are satisfied by the upper Furstenberg family of sets of positive upper density $\AA_{ud}$ defined in Example \ref{EX_A}. We recall that, in this case,
	\[D=\R_{+}^{*}\quad \text{ and }\quad
          \AA_{\delta,n}=\{A\subseteq \N: \, \exists N\geq n\text{
            such that } \frac{\#(A\cap [0,N])}{N+1}>\delta\}.\]	
It is easily seen that for every $k\geq0$, every $\delta'>\delta>0$, $A\in\AA_{\delta'}\ \implies\ A-k\in\AA_{\delta}$. Moreover, Bayart and Ruzsa \cite{Bayru} proved that if $A\in\AA_{\delta}$ then $\{k\geq 0:\ A\cap(A-k)\in\AA_{\frac{\delta^2}{2}}\}$ has bounded gaps. Hence, the conditions of Corollary \ref{Cor_AHC_nPROD} are fulfilled.
\end{proof}

We can also apply Corollary \ref{Cor_AHC_nPROD} to the family of sets having positive upper Banach density.

\begin{corollary}
	If $T$ is reiteratively hypercyclic, then for any $n\geq1$, the $n$-fold product $T\times\cdots\times T$ is reiteratively hypercyclic.
\end{corollary}

\begin{proof}
	We check the conditions of Corollary~\ref{Cor_AHC_nPROD} for the upper Furstenberg family $\AA=\AA_{uBd}$. We have
	\[D=\R_{+}^{*}\quad \text{ and }\quad
          \AA_{\delta,n}=\{A\subseteq \N: \, \exists m\geq 0\text{
            such that } \frac{\#(A\cap [m,m+n])}{n+1}>\delta\}.\] 
	Here again, for every $k\geq0$, every $\delta'>\delta>0$, $A\in\AA_{\delta'}\ \implies\ A-k\in\AA_{\delta}$ and the fact that if $A\in\AA_{\delta}$ then $\{k\geq
        0:\ A\cap(A-k)\in\AA_{\frac{\delta^2}{2}}\}$ has bounded gaps,
        has already been obtained, as Theorem 20 in  \cite{Besmenperpuig}.
\end{proof}

Let us now consider the last family we are interested in: the
family of sets having positive upper density with respect to some
weight $a\in\SS$. The result concerning the $n$-fold product of a
$\mathcal{U}$-frequently hypercyclic operator with respect to a weight
of $\SS$ will follow from Corollary \ref{Cor_AHC_nPROD} and the two
next lemmas. The first one gives the shift-invariance of the weighted
densities. 

\begin{lemma}\label{lem:1}
Let $A\subseteq \N$ and assume that $(n_i)_{i \in \N}$ is an
increasing sequence satisfying $\lim_{i \to +\infty}
\frac{\sum_{j=0}^{n_i}a_j\indi{A}(j)}{\sum_{j=0}^{n_i}a_j}=\delta$.
Then, for every $k \in \N$, one has $$
\lim_{i \to +\infty}\frac{\sum_{j=0}^{n_i}a_j\indi{A-k}(j)}{\sum_{j=0}^{n_i}a_j}=\delta
\quad \text{and}\quad \overline{d}_{a}(A) = \overline{d}_{a}(A-k)\, .$$
\end{lemma} 

\begin{proof}
%Let $A\subseteq \N$ and $(n_i)_i$ be an increasing sequence such that
%$\lim_i
%\frac{\sum_{j=0}^{n_i}a_j\indi{A}(j)}{\sum_{j=0}^{n_i}a_j}=\overline{d}_{a}(A)$. 
Remark that it suffices to prove the result for $k=1$. %  \begin{equation}\label{EQ_shift_inv}
	% \overline{d}_{a}(A)= 	\lim_{i \to +\infty} \frac{\sum_{j=0}^{n_i}a_j\indi{A-1}(j)}{\sum_{j=0}^{n_i}a_j}\quad\text{and}\quad \overline{d}_{a}(A) = \overline{d}_{a}(A-1).
	% 	\end{equation}
	Since $(a_k)_{k \in \N}$ is non-decreasing and tends to
        infinity, we have for every $N \in \N$
		\[
		\frac{\sum_{j=0}^{N}a_j\indi{A-1}(j)}{\sum_{j=0}^{N}a_j}=
\frac{\sum_{j=1}^{N+1}a_{j-1}\indi{A}(j)}{\sum_{j=0}^{N}a_j}
		\leq \frac{\sum_{j=0}^{N}a_{j}\indi{A}(j)}{\sum_{j=0}^{N}a_j}+\frac{a_{N+1}}{\sum_{j=0}^{N}a_j}.
		\]
		Hence, since $v_{N+1}(a)$ tends to zero for $a\in\SS$,
                we obtain that \[\limsup_{i \to  +\infty} \frac{\sum_{j=0}^{n_i}a_j\indi{A-1}(j)}{\sum_{j=0}^{n_i}a_j}\le \delta \quad \text{and} \quad \overline{d}_a(A-1)\le \overline{d}_a(A).\]
		
		For the reverse inequality, remember that for any $a\in\SS$, the sequence $(v_n(a))_{n\in\N}$ is non-increasing. This yields $a_n\leq a_{n-1}\frac{\sum_{j=0}^{n-1}a_j}{\sum_{j=0}^{n-2}a_j}=a_{n-1}(1+v_{n-1}(a))$ and from that we derive that for every $n_0\geq0$
		\begin{align*}
		\frac{\sum_{j=n_0}^{N}a_j\indi{A-1}(j)}{\sum_{j=0}^{N}a_j}&\ge\frac{\sum_{j=n_0+1}^{N}a_{j-1}\indi{A}(j)}{\sum_{j=0}^{N}a_j}\\
		&\geq \frac{\sum_{j=n_0+1}^{N}\frac{a_{j}}{1+v_{j-1}(a)}\indi{A}(j)}{\sum_{j=0}^{N}a_j}\\
		&\geq \frac{1}{1+v_{n_0}(a)}\frac{\sum_{j=n_0+1}^{N}a_j\indi{A}(j)}{\sum_{j=0}^{N}a_j}.
		\end{align*}
It follows that
\[
\liminf_{i \to + \infty} \frac{\sum_{j=0}^{n_i}a_j\indi{A-1}(j)}{\sum_{j=0}^{n_i}a_j}\ge \frac{1}{1+v_{n_0}(a)} \delta \quad \text{and}\quad \overline{d}_{a}(A-1) \geq  \frac{1}{1+v_{n_0}(a)} \overline{d}_{a}(A)
\]
and using  again that the sequence $(v_n(a))_{n \in \N}$ tends to
zero, we obtain the conclusion. %  reverse inequality: $\overline{d}_a(A+1)\leq \overline{d}_a(A)$.
		% So the shift invariance is proved.
\end{proof}

The proof of the second lemma follows the ideas of Theorem 8 from \cite{Bayru}. 

\begin{lemma}\label{lem:2}
Let $\delta >0$. If $A\subseteq \N$ satisfies $\overline{d}_{a}(A)>
\delta$, then the set \[\KK=\big\{k\in \N:\ \overline{d}_{a}
  \big(A\cap(A-k)\big) > \frac{\delta^2}{2}\big\}\]
		has bounded gaps.
\end{lemma}

\begin{proof}
Let us consider a sequence $(m_i)_{i\in\N}$ such
that: \[\frac{\sum_{j=0}^{m_i}a_j\indi{A}(j)}{\sum_{j=0}^{m_i}a_j}\underset{i\to\infty}{\longrightarrow}\overline{d}_a(A).\]
Since for every $k\in\N$, the sequence
$\left(\frac{\sum_{j=0}^{m_i}a_j\indi{A\cap(A-k)}(j)}{\sum_{j=0}^{m_i}a_j}
\right)_{i \in \N}$ belongs to the compact set $[0,1]$, we can use a
diagonal procedure to extract a subsequence of the sequence
$(m_i)_{i\in\N}$, which we will denote $(n_{i})_{i \in
  \N}$, such that for every $k \in \N$, the limit
 \[\frac{\sum_{j=0}^{n_i}a_j\indi{A\cap(A-k)}(j)}{\sum_{j=0}^{n_i}a_j}\underset{i\to\infty}{\longrightarrow}\delta_k\]
exists.  It suffices now to prove that the set $\mathcal{K}':=\{k\in
\N:\ \delta_k>\frac{\delta^2}{2}\}\subseteq \mathcal{K}$ has bounded gaps.
		Let $R$ be a finite set such that for every $k\neq l\in R$,
		\begin{equation}\label{EQ_Delta}
		\delta_{k-l}\leq\frac{\delta^2}{2}.
		\end{equation}
		On one hand, we have by Lemma \ref{lem:1}
		\[\frac{\sum_{j=0}^{n_i}a_j\left(\sum_{k\in R}\indi{A-k}(j)\right)}{\sum_{j=0}^{n_i}a_j}\underset{i\to\infty}{\longrightarrow}\sum_{k\in R} \overline{d}_a(A)=(\#R)\cdot \overline{d}_a(A).\]
	On the other hand,
	\begin{align*}
	\frac{\sum_{j=0}^{n_i}a_j\left(\sum_{k\in R}\indi{A-k}(j)\right)^2}{\sum_{j=0}^{n_i}a_j}&=\frac{\sum_{j=0}^{n_i}a_j\sum_{k\in R}\sum_{l\in R}\indi{(A-k)\cap(A-l)}(j)}{\sum_{j=0}^{n_i}a_j}\\
	&=\sum_{k\in R}\sum_{l\in R}\frac{\sum_{j=0}^{n_i}a_j\indi{(A-k)\cap(A-l)}(j)}{\sum_{j=0}^{n_i}a_j}\\
	&=\sum_{k\in R}\sum_{l\in R}\frac{\sum_{j=k}^{n_i+k}a_{j-k}\indi{A\cap(A+k-l)}(j)}{\sum_{j=0}^{n_i}a_j}
	\end{align*}
Using again Lemma \ref{lem:1}, we obtain that
	\[\frac{\sum_{j=k}^{n_i+k}a_{j-k}\indi{A\cap(A+k-l)}(j)}{\sum_{j=0}^{n_i}a_j}\underset{i\to\infty}{\longrightarrow}\delta_{l-k}.\]
	Therefore, we have 
	\[\frac{\sum_{j=0}^{n_i}a_j\left(\sum_{k\in R}\indi{A-k}(j)\right)^2}{\sum_{j=0}^{n_i}a_j}\underset{i\to\infty}{\longrightarrow}\sum_{k\in R}\sum_{l\in R}\delta_{l-k}\leq (\#R)(\#R-1)\frac{\delta^2}{2}+(\#R)\overline{d}_a(A).\]
	If we define $f(j)=\sum_{k\in R}\indi{A-k}(j)$ then Jensen inequality gives \[\mathbb{E}_{a}^{i}(f^2)\geq \left(\mathbb{E}_{a}^{i}(f)\right)^2,\]
	where $\mathbb{E}_{a}^{i}(X):=\sum_{j=0}^{n_i}\frac{a_j}{\sum_{n=0}^{n_i}a_n}X(j)$.
	Considering what has been proved before, this yields
	\[(\#R)(\#R-1)\frac{\delta^2}{2}+(\#R)\overline{d}_a(A)\geq \left((\#R)\overline{d}_a(A)\right)^2.\]
	After having used the assumption $\overline{d}_a(A)>\delta$ and simplified the previous equation, we obtain:
	\[\#R\leq
          \frac{\overline{d}_a(A)-\frac{\delta^2}{2}}{\left(\overline{d}_{a}(A)\right)^2-\frac{\delta^2}{2}}.\]
Hence, one can choose a set $R$ being maximal for relation
(\ref{EQ_Delta}). By maximality, if $n \notin R$, there exists $k \in
R$ such that $\delta_{n-k}> \frac{\delta^2}{2}$. It follows that $\mathcal{K}'+R=\N$, which gives that $\mathcal{K}'$ has bounded gaps.
\end{proof}

We can now easily obtain the result concerning
$\mathcal{U}$-frequently hypercyclic operators with respect to $a\in\SS$.

\begin{corollary}
	If $T$ is $\mathcal{U}$-frequently hypercyclic with respect to $a\in\SS$, then for any $n\geq1$, the $n$-fold product $T\times\cdots\times T$ is $\mathcal{U}$-frequently hypercyclic with respect to $a$.
\end{corollary}

\begin{proof}
		Let $a\in \SS$. We consider the upper Furstenberg
                family $\AA_{ud_{a}}$ defined in Example \ref{EX_A}
                and the conclusion is obtained by combining
                Corollary~\ref{Cor_AHC_nPROD}, Lemma~\ref{lem:1} and
                Lemma~\ref{lem:2}. 
\end{proof}

Recall that it is not known if $T\times T$ is frequently hypercyclic as soon as $T$ is frequently hypercyclic. However, it is shown in \cite{Menet2019} that if we let $\mathcal{B}:=\{(b_n)_{n\in\N}:\ (b_n)_{n\in\N}\text{ is non-increasing, tends to zero and }\sum_{n\in\N} b_n=\infty\}$, then for every set $A$, $\inf_{b\in \mathcal{B}}\overline{d}_b(A)=\underline{d}(A)$ and in particular, for every operator $T$, we have $FHC(T)=\bigcap_{b\in\mathcal{B}}UFHC_b(T)$. In other words, this family of weighted upper densities allows to approach the lower density and it is interesting to notice that by following similar arguments, we get the same result for weights in $\mathcal{B}$. 

\begin{corollary}
If $T$ is $\mathcal{U}$-frequently hypercyclic with respect to $b\in\mathcal{B}$, then for any $n\geq1$, the $n$-fold product $T\times\cdots\times T$ is $\mathcal{U}$-frequently hypercyclic with respect to $b$.
\end{corollary}

These remarks lead to the following natural questions.

\begin{question}
	Is the $n$-fold product $T\times\cdots\times T$ of a
        frequently hypercyclic operator $T$ also frequently
        hypercyclic? %Note that in \cite{Menet2019}, it is proved that the
       % weighted densities $\overline{d}_b$ with $b \in \mathcal{B}$
       % approach the lower density $\underline{d}$.
\end{question}

\begin{question}
	What can be said about an infinite product of $T$? A variant
        of this question has been studied in a paper by Grivaux,
        Matheron and Menet \cite{Monster} where they obtain negative
        results for some infinite products of all different $\mathcal{U}$-frequently hypercyclic operators.
\end{question}

\section{Chaos and $\mathcal{U}$-frequent hypercyclicity with respect to $a\in\SS$}

The aim of this section is to study the link between chaos and
frequent hypercyclicity with respect to weights in $\SS$. Indeed, we
have seen in Theorem \ref{TH_eq_UFHCa_RHCa} that
$\mathcal{U}$-frequent hypercyclicity with respect to weights in $\SS$
is an intermediate notion lying between $\mathcal{U}$-frequent
hypercyclicity and reiterative hypercyclicity. In addition, Menet
\cite{MenetChaos} proved that every chaotic operator is automatically
reiteratively hypercyclic while there exist chaotic operators that
are not $\mathcal{U}$-frequently hypercyclic. These results induce the
question of knowing whether there exists  weights $a \in \SS$ for
which chaotic operators are always $\mathcal{U}$-frequently
hypercyclic with respect to $a$. In what follows, this question is
answered negatively by considering on $\ell^{1}(\mathbb{N})$ an operator  of C-type as introduced in
\cite{MenetChaos,Monster}. 

\begin{definition}
The operator of $C$-type $T_{v,w,\varphi,b}$ on $\ell^{p}(\mathbb{N})$ is defined
by 
\[
T_{v,w,\varphi,b}e_{k} =
\begin{cases}
w_{k+1}e_{k+1}& \text{if } k \in [b_{n},b_{n+1}-1[ \, , \, n \ge 0\\[2.5ex]
v_{n}e_{b_{\varphi(n)}} - \Big(
  \displaystyle\prod_{j=b_{n}+1}^{b_{n+1}-1}w_{j}\Big)^{-1}e_{b_n}&  \text{if }
                                                         k =
                                                         b_{n+1}-1, \,
                                                         n \ge 1\\[2.5ex]
 - \Big( \displaystyle \prod_{j=b_{0}+1}^{b_{1}-1}w_{j}\Big)^{-1}e_{0}&  \text{if }
  k = b_{1}-1
\end{cases}
\]
where the four parameters $v, w , \varphi$ and $b$ are chosen as follows:
\begin{itemize}
\item[-] $v=(v_{n})_{n \ge 1}$ is a sequence of non-zero complex numbers
  such that $\sum_{n \geq 1}|v_{n}|< + \infty$,
\item[-]  $w=(w_{n})_{n \ge 1}$ is a sequence of complex numbers such that
  $0 < \inf_{n \ge 1}|w_{n}|\leq \sup_{n \ge 1}|w_{n}|<  + \infty $,
\item[-]  $\varphi : \N \to \N$ satisfies $\varphi(0)=0$, $\varphi(n) < n$
  and $\#\varphi^{-1}(\{n\}) = + \infty$  for every $n \ge 1$,
\item[-]  $b=(b_{n})_{n \ge 1}$ is an increasing sequence of integers such
  that $b_{0}=0$ and $b_{n+1}-b_{n}$ is a multiple of
  $2(b_{\varphi(n)+1}- b_{\varphi(n)})$ for every $n \ge 1$. 
\end{itemize}
Moreover, we assume that $\inf_{n \ge 0}\prod_{j=b_{n}+1}^{b_{n+1}-1}|w_{j}|>0$. 
\end{definition}

Note that every basis vector $e_{k}$ is periodic for $T_{v,w,\varphi,b}$, hence
the same holds for  every finite vector $x \in c_{00}$. 

Let us begin by stating a result which is a direct consequence of the
proof of Lemma~6.11 from \cite{Monster}. Indeed, the following lemma
is proved there but not explicitly stated. The operator $P_{n}$
denotes the canonical projection of $\ell^{p}$ onto $\text{span}\{
e_{k}: b_{n}\leq k < b_{n+1}\}$. 

\begin{lemma}\label{Lem_6.11}
	Let $T$ be an operator of C-type and
        $x\in\ell^p(\N)\setminus\{0\}$. Assume that there exist a constant $C>0$, a non-increasing
        sequence $(\beta_l)_{l\in\N}$ of positive real numbers
        satisfying $\sum_{l\in\N}\sqrt{\beta_l}\leq 1$, a sequence
        $(X_l)_{l \in \N}$ of non-negative real numbers and a
        sequence $(N_l)_{l\in\N}$ of integers increasing to infinity such that
	\begin{enumerate}[(i)]
		\item $\Vert P_nx\Vert\leq X_n$, for every $n\in\N$,\label{H1}
		\item $\sup_{j\in\N}\Vert P_nT^jP_lx\Vert\leq C\beta_l X_l$ for $0\leq n<l$,\label{H2}
		\item $\sup_{0\leq j\leq N_l}\Vert P_nT^j P_lx\Vert\leq C\beta_l\Vert P_lx\Vert$ for $0\leq n\leq l$,\label{H3}
		\item  $\sup_{j\in\N}\sum_{l>n}\Vert P_nT^jP_lx\Vert> C X_n$ for every $n\in\N$.\label{H4}
	\end{enumerate}
Then, there exist $\varepsilon>0$ and $(l_{n_{s}})_{s\in\N}$ increasing to infinity such that for every $s\in\N$,
\begin{itemize}
	\item $N\left(x,B(0,\varepsilon)^c\right)\cap [0,s]\supseteq \{0\leq j\leq s:\ \Vert P_{l_{n_{s}}}T^j P_{l_{n_{s}}}x\Vert\geq 2CX_{l_{n_{s}}}\}$,
	\item $s>N_{l_{n_{s}}}$.
\end{itemize}
\end{lemma}

We present now, our main criterion to check that the operator we are going to construct is not $\mathcal{U}$-frequently hypercyclic with respect to $a$. 

\begin{theorem}\label{TH_C_type_not_UFHC_a}
	Let $T$ be an operator of C-type and $a\in\SS$. Assume that for every hypercyclic vector $x$, there exist a constant $C>0$, a non-increasing
        sequence $(\beta_l)_{l\in\N}$ of positive real numbers
        satisfying $\sum_{l\in\N}\sqrt{\beta_l}\leq 1$, a sequence
        $(X_l)_{l \in \N}$ of non-negative real numbers and a
        sequence $(N_l)_{l\in\N}$ of integers increasing to infinity satisfying (\ref{H1}), (\ref{H2}) and (\ref{H3}) from Lemma \ref{Lem_6.11} and
	\begin{equation}\label{H5}
	\liminf_{l\to + \infty}\inf_{s\geq
          N_l}\frac{\sum_{j=0}^{s}a_j\indi{\{j\in \N:\ \Vert P_lT^jP_lx\Vert\geq 2CX_l\}}(j)}{\sum_{j=0}^{s}a_j}=1.\tag{C}
	\end{equation}
	Then $T$ is not $\mathcal{U}$-frequently hypercyclic with respect to $a$.
\end{theorem}

\begin{proof}
	Let us first remark that since $x$ is a hypercyclic vector for
        $T$, one has  \[\sup_{j\geq0}\sum_{l>n}\Vert P_n T^j
          P_l\,x\Vert =+\infty\] 
for every $n\in\N$. Indeed, since $x$ is hypercyclic, we have
$\sup_{j\geq0}\Vert P_n T^j x\Vert=\infty$ for every $n\in \N$ and
since $P_nT^jx=\sum_{l\ge n} P_n T^jP_l\,x$, we get the desired
equality by remarking that $\sup_{j\geq0} \Vert P_n
T^jP_nx\Vert<\infty$ because $P_n x$ is a periodic point. It follows that condition (\ref{H4}) from Lemma
\ref{Lem_6.11} is satisfied for any choice of sequence $(X_l)_{l \in \N}$.
	Thus, applying Lemma \ref{Lem_6.11}, there exists $\varepsilon>0$ and $(l_{n_{s}})_{s\in\N}$ increasing to infinity such that for every $s\in\N$,
	\begin{itemize}
		\item $N\left(x,B(0,\varepsilon)^c\right)\cap [0,s]\supseteq \{0\leq j\leq s:\ \Vert P_{l_{n_{s}}}T^j P_{l_{n_{s}}}x\Vert\geq 2CX_{l_{n_{s}}}\}$,
		\item $s>N_{l_{n_{s}}}$.
	\end{itemize}
	Let us use these two consequences in the computation of the lower $a$-density of $N(x,B(0,\varepsilon)^c)$:
	\begin{align*}
	\frac{\sum_{j=0}^{s}a_j\indi{N(x,B(0,\varepsilon)^c)}(j)}{\sum_{j=0}^{s}a_j}&\geq \frac{\sum_{j=0}^{s}a_j\indi{\{j:\ \Vert P_{l_{n_{s}}}T^jP_{l_{n_{s}}}x\Vert\geq 2CX_{l_{n_{s}}}\}}(j)}{\sum_{j=0}^{s}a_j}\\
	&\geq \inf_{r>N_{l_{n_{s}}}} \frac{\sum_{j=0}^{r}a_j\indi{\{j:\ \Vert P_{l_{n_{s}}}T^jP_{l_{n_{s}}}x\Vert\geq 2CX_{l_{n_{s}}}\}}(j)}{\sum_{j=0}^{r}a_j}\\
	&\geq \inf_{l\geq l_{n_{s}}}\inf_{r>N_{l}} \frac{\sum_{j=0}^{r}a_j\indi{\{j:\ \Vert P_{l}T^jP_{l}x\Vert\geq 2CX_{l}\}}(j)}{\sum_{j=0}^{r}a_j}
	\end{align*}
	Hence, we deduce
	% \[\underline{d}_a(N(x,B(0,\varepsilon)^c))\geq\liminf_{l\to+\infty}\inf_{r>N_{l}}
        %   \frac{\sum_{j=0}^{r}a_j\indi{\{j:\ \Vert
        %       P_{l}T^jP_{l}x\Vert\geq
        %       2CX_{l}\}}(j)}{\sum_{j=0}^{r}a_j}=1.\]
	\[\liminf_{s \to + \infty }\frac{\sum_{j=0}^{s}a_j\indi{N(x,B(0,\varepsilon)^c)}(j)}{\sum_{j=0}^{s}a_j}
          \ge \liminf_{l\to+\infty}\inf_{r>N_{l}} \frac{\sum_{j=0}^{r}a_j\indi{\{j:\ \Vert P_{l}T^jP_{l}x\Vert\geq 2CX_{l}\}}(j)}{\sum_{j=0}^{r}a_j}=1\]
	which in turn implies that
        $\overline{d}_a(N(x,B(0,\varepsilon)))=0$. Consequently $T$ cannot be $\mathcal{U}$-frequently hypercyclic with respect to $a$.
\end{proof}

The following lemma will be helpful for the proof of our main result. It is a slight modification of Fact 6.13 from \cite{Monster}. The proof is only a rewriting of the original proof but we keep the sets that we consider in the proof instead of counting their elements like it was done in \cite{Monster}.

\begin{lemma}\label{Lem_Fact_6.13}
	Let $T$ be an operator of C-type and $x\in\ell^p(\N)$.
	Let also \[X_l=\Big\Vert
          \sum_{k=b_l}^{b_{l+1}-1}\left(\prod_{s=k+1}^{b_{l+1}-1}w_s\right)x_k
          e_k\Big\Vert \, .\]
	If there exists $0\leq k_0<k_1 \leq b_{l+1}-b_l$ such that
	\[\vert w_{b_l+k}\vert=1\text{, for every }k\in\, ]k_0,k_1[\, \text{ and }\prod_{s=b_l+k_0+1}^{b_{l+1}-1}\vert w_s\vert=1,\]
	then there exists $j_0<b_{l+1}-b_l$ such that for every $J\in\N$,
	\begin{multline*}
	\left\{0\leq j\leq J:\ \Vert
          P_lT^jP_lx\Vert<\frac{X_l}{2}\right\}\\ \subseteq[0,J]\cap
	\bigcup_{r\in\N}\big[r(b_{l+1}-b_l)+j_0,r(b_{l+1}-b_l)+j_0+2(b_{l+1}-b_l-(k_1-k_0))\big].
	\end{multline*}
\end{lemma}

The operator we are looking for will be constructed as an
operator of $C_{+,1}$-type as defined by Grivaux, Matheron and Menet in
Section 6.5 from \cite{Monster}, which are particular cases of
operators of $C$-type. These operators depend only on three increasing
sequences of integers $(\delta^{(k)})_{k \in \N}, (\Delta^{(k)})_{k \in
  \N}$ and $(\tau^{(k)})_{k \in \N}$ satisfying $\delta^{(k)}<
\Delta^{(k)}$ for every $k \in \N$. The parameters $v,w,\varphi$ and
$b$ are then defined by 
\begin{itemize} 
\item[-] $v_{n} = 2^{-\tau^{(k)}}$ for every $n \in [2^{k-1},
  2^{k}\,[$, 
\item[-] $b_{0}=0$ and $b_{n+1}= b_{n}+ \Delta^{(k)}$ for every $n \in [2^{k-1}, 2^{k}\,[$, 
\item[-] $w_{b_{n}+i}= w_{i}^{(k)}$ for every $n \in [2^{k-1},
  2^{k}\,[$ and every $i \in [1, \Delta^{(k)}[\,$, where \[w_{i}^{(k)}=\begin{cases}
	2&\text{ if }1\leq i\leq \delta^{(k)}\\
	1&\text{ if }\delta^{(k)}< i< \Delta^{(k)}
	\end{cases}\]
\item[-] $\varphi(n) = n - 2^{k-1}$ for every $n \in [2^{k-1},
  2^{k}\,[$. 
\end{itemize}
	% \[v^{(k)}=2^{-\tau^{(k)}}\text{ and }w_{i}^{(k)}=\begin{cases}
	% 2&\text{ if }1\leq i\leq \delta^{(k)},\\
	% 1&\text{ if }\delta^{(k)}< i< \Delta^{(k)}.
	% \end{cases}\]
	% Moreover, in the formal definition of $C$-type operators, one also needs a sequence $(b_n)_{n\in\N}$ but in the case of $C_{+}$-type operators, this sequence is given by the choice of the sequence $(\Delta^{(k)})_k$ since $b_0=0$ and $b_{n+1}-b_n=\Delta^{(k)}$ when $n\in[2^{k-1};2^k[$.
%Note that the quantity $\frac{\delta^{(k)}}{\Delta^{(k)}}$ gives
%proposition of weights equal to $2$ in each block
%$[b_{n},b_{n+1}[\,$. 
The result stated below, see Theorem 6.17 of \cite{Monster}, gives a
sufficient condition for an operator of $C_{+,1}$-type to be chaotic.

\begin{theorem}\label{Thm_6.17}
Let $T$ be a $C_{+,1}$-type operator on $l^{p}(\N)$. If $\limsup_{k \to +
  \infty}(\delta^{(k)}- \tau^{(k)})= + \infty$, then $T$ is chaotic. 
\end{theorem}

The last result we will need gives sufficient conditions on the sequences $(\delta^{(k)})_{k \in \N}$ and $(\tau^{(k)})_{k \in \N}$ to construct sequences that
satisfies the assumptions of Lemma \ref{Lem_6.11}.  It has been
obtained within the beginning of the proof of Theorem 6.18 in \cite{Monster}.

\begin{lemma}\label{Lem_6.18}
Assume that the sequences $(\delta^{(k)})_{k \in \N}$ and
$(\tau^{(k)})_{k \in \N}$ satisfy 
\[\limsup_{k\to+\infty}\frac{\tau^{(k)}}{\delta^{(k)}}<1 \quad
  \text{and} \quad
  \sum_{k\geq1} 2^k2^{\frac{\delta^{(k-1)}-\tau^{(k)}}{2}}(\Delta^{(k)})^{\frac{1-\frac{1}{p}}{2}} \leq1 \, \]
  and that the sequence $(2^{\delta^{(k-1)}-\tau^{(k)}}(\Delta^{(k)})^{1-\frac{1}{p}})$ is non-increasing.
Then hypotheses (\ref{H1}), (\ref{H2}) and (\ref{H3}) of Lemma \ref{Lem_6.11} are satisfied with the following choices: 
		\begin{itemize}
			\item $C=\frac{1}{4}$,
			\item $\beta_l=4\cdot 2^{\delta^{(k-1)}-\tau^{(k)}}(\Delta^{(k)})^{1-\frac{1}{p}}$ when $l\in[2^{k-1},2^k[$,
			\item $X_l=\Vert \sum_{k=b_l}^{b_{l+1}-1}\left(\prod_{s=k+1}^{b_{l+1}-1}w_s\right)x_ke_k\Vert$ when $l\in[2^{k-1},2^k[$,
			\item $N_l=\Delta^{(k)}-\delta^{(k)}$ when  $l\in[2^{k-1},2^k[$.
		\end{itemize}
\end{lemma}

This preparatory being done, we  can now state and prove the main result of this section which proves that there is no link between $\mathcal{U}$-frequent hypercyclicity with respect to any sequences in $\mathcal{S}$ and chaos.

\begin{theorem}\label{TH_chaos_not_ufhca}
	For any $a\in\SS$, there exists a chaotic operator on $\ell^1(\mathbb{N})$ which is not $\mathcal{U}$-frequently hypercyclic with respect to $a$.
\end{theorem}

\begin{proof}
	
	% First, by Theorem 6.17 from \cite{Monster}, if $\delta^{(k)}-\tau^{(k)}\underset{k\to+\infty}{\longrightarrow}+\infty$ then $T$ is chaotic.
	% On the other hand, if $\limsup_{k\to+\infty}\frac{\tau^{(k)}}{\delta^{(k)}}<1$ and $\sum_{k\geq1} 2^k\gamma_{k}^{\frac{1}{2}}\leq1$ where $\gamma_k:=2^{\delta^{(k-1)}-\tau^{(k)}}$ then hypotheses (\ref{H1}), (\ref{H2}) and (\ref{H3}) from Lemma \ref{Lem_6.11} are satisfied with the following choices: 
	% 	\begin{itemize}
	% 		\item $C=\frac{1}{4}$,
	% 		\item $\beta_l=4\gamma_k$ when $l\in[2^{k-1};2^k[$,
	% 		\item $X_l=\Vert \sum_{i=b_l}^{b_{l+1}-1}\left(\prod_{s=i+1}^{b_{l+1}-1}w_s\right)x_ie_i\Vert$ when $l\in[2^{k-1};2^k[$,
	% 		\item $N_l=\Delta^{(k)}-\delta^{(k)}$ when  $l\in[2^{k-1},2^k[$.
	% 	\end{itemize}
	% Indeed, this was already proved at the beginning of the
        % proof of Theorem 6.18 from \cite{Monster}.
% Using Theorem \ref{Thm_6.17} and Lemma \ref{Lem_6.18}, it suffThis preliminary work being done, we define $\tau^{(k)}=\frac{1}{2}\delta^{(k)}$ and in order to apply Theorem \ref{TH_C_type_not_UFHC_a}, we want to construct
First, let us fix a sequence $(\delta^{(k)})_{k \in \N}$ such that $\frac{1}{2}\delta^{(k)}-\delta^{(k-1)}$ is increasing and
	\begin{equation}\label{eqC+2}
2^{\delta^{(k-1)}-\frac{1}{2}\delta^{(k)}}\leq\frac{1}{4^{2k}}
	\end{equation}
and let us define the sequence $(\tau^{(k)})_{k \in \N}$ by setting
$\tau^{(k)}=\frac{1}{2}\delta^{(k)}$. Using Theorem~\ref{Thm_6.17}, it
is clear that for any choice of sequence $(\Delta^{(k)})_{k \in \N}$,
the operator will be chaotic. Moreover, the assumptions of
Lemma \ref{Lem_6.18} are clearly fulfilled since $p=1$. It gives the existence of a
constant $C>0$ and of sequences $(\beta_l)_{l\in\N}$, $(X_l)_{l \in
  \N}$ and  $(N_l)_{l\in\N}$ satisfying hypotheses (\ref{H1}),
(\ref{H2}) and (\ref{H3}) of Lemma \ref{Lem_6.11}, where $C=\frac{1}{4}$ and the sequence $N_l=\Delta^{(k)}-\delta^{(k)}$ when  $l\in[2^{k-1},2^k[$. 

It remains to show that for a good choice of $(\Delta^{(k)})_{k \in \N}$, condition~(\ref{H5}) is satisfied so that by invoking Theorem~\ref{TH_C_type_not_UFHC_a}, we can deduce that $T$ is not $\mathcal{U}$-frequently hypercyclic with respect to $a$. 

We first remark that by definition of operators of $C_{+,1}$-type, for every $i\in\, ]\delta^{(k)},\Delta^{(k)}[$, one has
$w_{i}^{(k)}=1$ hence $\prod_{i=\delta^{(k)}+1}^{\Delta^{(k)}-1}w_{i}^{(k)}=1$. Consequently,
Lemma~\ref{Lem_Fact_6.13} can be applied with $k_0=\delta^{(k)}$ and
$k_1=\Delta^{(k)}$ when $l\in[2^{k-1},2^k[\,$. We get that for every $k\ge 1$, every $l\in [2^{k-1},2^{k}[$, there exists $j_0 < \Delta^{(k)}$ such that for every $J\in\N$,
\[	\left\{0\leq j\leq J:\ \Vert
          P_lT^jP_lx\Vert<\frac{X_l}{2}\right\} \subseteq[0,J]\cap
	\bigcup_{r\in\N}\big[r\Delta^{(k)}+j_0,r\Delta^{(k)}+j_0+2\delta^{(k)}\big].
\]
Therefore, it suffices to find a sequence $(\Delta^{(k)})_{k \in \N}$ satisfying for every $k\in\N$
	\begin{equation}\label{eqC+1}
	\lim_{k\to +\infty}\max_{j_0<\Delta^{(k)}}\sup_{s\geq\Delta^{(k)}-\delta^{(k)}}\frac{\sum_{j=0}^{s}a_j\indi{\cup_{r\in\N}[r\Delta^{(k)}+j_0,r\Delta^{(k)}+j_0+2\delta^{(k)}]}(j)}{\sum_{j=0}^{s}a_j}=0
 	\end{equation}
in order to conclude that condition~(\ref{H5}) is satisfied:
\[
	\liminf_{l\to + \infty}\inf_{s\geq
          N_l}\frac{\sum_{j=0}^{s}a_j\indi{\{j\in \N:\ \Vert P_lT^jP_lx\Vert\geq \frac{X_l}{2}\}}(j)}{\sum_{j=0}^{s}a_j}=1.
\]
% 	and 
% 	\begin{equation}\label{eqC+2}
% 2^{\delta^{(k-1)}-\frac{1}{2}\delta^{(k)}}\leq\frac{1}{4^{2k}}.
% 	\end{equation}
% The existence of some $(\delta^{(k)})_{k \in \N}$ satisfying (\ref{eqC+2}) is clear, so let us fix such a sequence.
In other words, it suffices to prove that for every $n\in\N$,
	\[\lim_{N\to+\infty}\max_{j_0<N}\sup_{s\geq N-n}\frac{\sum_{j=0}^{s}a_j\indi{\cup_{r\in\N}[rN+j_0,rN+j_0+2n]}(j)}{\sum_{j=0}^{s}a_j}=0.\]
	Let us divide this proof into two cases.
	\begin{itemize}
		\item If $N-n\leq s<N$, since $(a_n)_{n \in
                  \N}\in \SS$ and in particular since $v_j(a)$ is non-increasing, we get
		\begin{align*}
		\max_{j_0<N}\frac{\sum_{j=0}^{s}a_j\indi{\cup_{r\in\N}[rN+j_0,rN+j_0+2n]}(j)}{\sum_{j=0}^{s}a_j}&\leq \frac{\sum_{j=s-2n}^{s}a_j}{\sum_{j=0}^{s}a_j}\\
		&=1-\frac{\sum_{j=0}^{s-2n-1}a_j}{\sum_{j=0}^{s}a_j}\\
		&=1-\frac{1}{\prod_{j=s-2n}^{s}(1+v_j(a))}\\
		&\leq 1-\frac{1}{\prod_{j=N-3n}^{N-n}(1+v_j(a))}.
		%&\underset{N\to+\infty}{\longrightarrow}0.
		\end{align*}
		\item If $r_{0}N\leq s<(r_{0}+1)N$ for $r_{0}\geq1$, then
		\[\max_{j_0<N}\frac{\sum_{j=0}^{s}a_j\indi{\cup_{r\in\N}[rN+j_0,rN+j_0+2n]}(j)}{\sum_{j=0}^{s}a_j}\leq
                  \frac{\sum_{j=s-2n}^{s}a_j+\sum_{r=1}^{r_{0}}\sum_{j=rN-2n}^{rN}a_j}{\sum_{j=0}^{s}a_j}.\]
		For the first term in the right hand side formula, we have
		\[\frac{\sum_{j=s-2n}^s a_j}{\sum_{j=0}^s a_j}\le 1-\frac{1}{\prod_{j=s-2n}^{s}(1+v_j(a))} \le 1-\frac{1}{\prod_{j=N-3n}^{N-n}(1+v_j(a))}\]
as we have already proved in the previous case.
		Let us now decompose the second term as the whole sum and its last term. 
		% \[\frac{\sum_{k=1}^{r}\sum_{j=kN-n}^{kN}a_j}{\sum_{j=0}^{s}a_j}\]
		We begin by the last term:
		\begin{align*}
		\frac{\sum_{j=r_0 N-2n}^{r_{0}N}a_j}{\sum_{j=0}^{s}a_j}&\leq \frac{\sum_{j=r_0N-2n}^{r_{0}N}a_j}{\sum_{j=0}^{r_0N}a_j}\\
		&\leq1-\frac{1}{\prod_{j=r_{0}N-2n}^{r_{0}N}(1+v_j(a))}\\
		&\leq
                  1-\frac{1}{\prod_{j=N-3n}^{N-n}(1+v_j(a))}.% \underset{N\to\infty}{\longrightarrow}0.
		\end{align*}

		The other terms can be treated in the following fashion
		\begin{align*}		
		\frac{\sum_{r=1}^{r_{0}-1}\sum_{j=rN-2n}^{rN}a_j}{\sum_{j=0}^{s}a_j}&\leq \frac{\sum_{r=1}^{r_{0}-1}(2n+1)a_{rN}}{\sum_{j=0}^{s}a_j}\\
		&\leq \frac{2n+1}{N}\frac{\sum_{r=1}^{r_{0}-1}\sum_{j=rN}^{(r+1)N-1}a_j}{\sum_{j=0}^{s}a_j}\\
		&\leq \frac{2n+1}{N}.%\underset{N\to\infty}{\longrightarrow}0.
		\end{align*}
	\end{itemize}
In summary, we have shown that
\[\max_{j_0<N}\sup_{s\ge N-n}\frac{\sum_{j=0}^{s}a_j\indi{\cup_{r\in\N}[rN+j_0,rN+j_0+2n]}(j)}{\sum_{j=0}^{s}a_j}\le 2-\frac{2}{\prod_{j=N-3n}^{N-n}(1+v_j(a))}+\frac{2n+1}{N}\]
which tends to $0$ as $N\to +\infty$. This concludes the proof.
\end{proof}

% \textcolor{red}{ou doit-on se restreindre a $\ell^{1}$??}\textcolor{blue}{ ??? }\textcolor{green}{Le lemme 3.6 était énoncé dans le cas de $\ell^1$. Dans le cas de $\ell^p$, la suite $\Delta^{(k)}$ apparaît dans les conditions (voir nouvel énoncé). Pour $p>1$, nous ne pouvons des lors pas choisir $\Delta^{(k)}$ arbitrairement grand à chaque étape comme cela est fait dans la preuve du Théorème 3.7.}

%\printbibliography
\bibliographystyle{plain}
\bibliography{mybiblio}

\end{document}